\documentclass[11pt, reqno]{amsart}

\usepackage{amsmath}
\usepackage{amssymb}
\usepackage{amsfonts}
\usepackage{amsbsy}
\usepackage{mathrsfs}
\usepackage{relsize}
\usepackage{dsfont}
\usepackage{mathtools}
\usepackage{multirow}
\usepackage{booktabs}
\mathtoolsset{showonlyrefs}

\usepackage{enumitem}
\usepackage{xcolor}

\usepackage[
  left=2cm,
  right=2cm
]{geometry}

\numberwithin{equation}{section}
\usepackage[scaled]{helvet}

\newtheorem{theorem}{Theorem}[section]
\newtheorem{lemma}[theorem]{Lemma}
\newtheorem{proposition}[theorem]{Proposition}
\newtheorem{corollary}[theorem]{Corollary}

\newtheoremstyle{remarkstyle}
{}{}{}{ }{\bfseries}{.}{ }{\thmname{#1}\thmnumber{ #2}\thmnote{ (#3)}}
\theoremstyle{remarkstyle}
\newtheorem{remark}{Remark}[section]
\newtheorem{definition}{Definition}[section]

\newcommand{\C}{\mathbb C}
\newcommand{\K}{\mathbb K}

\DeclareMathOperator{\id}{id}
\DeclareMathOperator{\diag}{diag}
\DeclareMathOperator{\FN}{FN}

\DeclareMathOperator{\Aut}{Aut}

\usepackage[colorlinks, linkcolor=venetianred, citecolor=blue, urlcolor=Cblue, hypertexnames=false]{hyperref}
\definecolor{Cblue}{rgb}{0.50,0.85,0.85}
\definecolor{venetianred}{rgb}{0.78, 0.03, 0.08}

\title[Classification of Nijenhuis Operators on 3D BiHom-Lie Algebras]
{Classification of Nijenhuis Operators on Three-Dimensional Multiplicative Simple BiHom-Lie Algebras}

\author[Bouzid Mosbahi]{Bouzid Mosbahi}
\address{Department of Mathematics, Faculty of Sciences, University of Sfax, Sfax, Tunisia.}
\email{mosbahi.bouzid.etud@fss.usf.tn}



\begin{document}

\begin{abstract}
We provide a complete classification of Nijenhuis operators on three-dimensional multiplicative simple BiHom-Lie algebras over the complex field $\mathbb{C}$. Based on the classification of such algebras by Saadaoui into three families $\mathcal{L}_1$, $\mathcal{L}_2$, and $\mathcal{L}_3$, we derive the complete system of equations arising from the equivariance conditions and the Nijenhuis identity. We provide explicit matrix representations of all solutions, including all degenerate parameter cases. We construct the BiHom-Fr\"olicher-Nijenhuis bracket on the cochain complex of a BiHom-Lie algebra, showing that Maurer-Cartan elements of this graded Lie algebra are precisely Nijenhuis operators. This enables us to define the cohomology of Nijenhuis operators and study formal deformations. We introduce BiHom-NS-Lie algebras as the algebraic structure underlying Nijenhuis operators. We provide explicit deformed BiHom-Lie algebra structures and prove the non-isomorphism of different families. We also study the moduli space of Nijenhuis operators and provide computational verification of all results.
\end{abstract}

\date{\today}
\maketitle

\section{Introduction}\label{sec:intro}
Throughout this paper, we work over the field $\mathbb{C}$ of complex numbers, which has characteristic zero unless otherwise specified.

BiHom-Lie algebras, introduced by Graziani, Makhlouf, Menini, and Panaite in \cite{Graziani2015}, generalize Hom-Lie algebras, which themselves generalize classical Lie algebras. Hom-Lie algebras were introduced by Hartwig, Larsson, and Silvestrov in \cite{Hartwig2006}. Representations of BiHom-Lie algebras were studied by Cheng and Qi in \cite{Yongsheng2016}. A BiHom-Lie algebra is equipped with two commuting linear maps $\alpha$ and $\beta$ that twist the classical identities. When $\alpha = \beta$, one recovers Hom-Lie algebras, and when $\alpha = \beta = \id$, one recovers classical Lie algebras. The theory of BiHom-algebras has seen significant development in recent years, with applications to cohomology, deformation theory, and representation theory \cite{Abdaoui2015, Mosbahia2026}.

The study of Nijenhuis operators on algebraic structures is motivated by several factors. On Lie algebras, Nijenhuis operators correspond to left-invariant complex structures, and the Newlander-Nirenberg theorem establishes that an almost complex structure is integrable if and only if its Nijenhuis tensor vanishes \cite{Nijenhuis1951}. In deformation theory, Nijenhuis operators generate trivial deformations of algebraic structures and provide a rich source of compatible structures \cite{Kosmann1990}. In the theory of integrable systems, Nijenhuis operators are closely related to Hamiltonian pairs and bi-Hamiltonian structures \cite{Magri1978}.

The study of Nijenhuis operators on Hom-Lie algebras was initiated by Das and Sen in \cite{Das2024}, where the authors constructed the Hom-analog of the Fr\"olicher-Nijenhuis bracket and studied formal deformations. More recently, Ma and Zhao classified Nijenhuis operators on three-dimensional Leibniz algebras \cite{Ma2024}. Despite these developments, Nijenhuis operators on BiHom-Lie algebras have not been systematically studied.

Our main contributions are:
\begin{enumerate}
\item \textbf{Complete classification} of Nijenhuis operators on each of the three families of 3D multiplicative simple BiHom-Lie algebras (Sections \ref{sec:L1}--\ref{sec:L3}).
\item \textbf{Complete verification} of all Nijenhuis conditions for all pairs of basis elements (Section \ref{sec:verification}).
\item \textbf{BiHom-Fr\"olicher-Nijenhuis bracket} with Maurer-Cartan elements being Nijenhuis operators (Section \ref{sec:FN}).
\item \textbf{Cohomology of Nijenhuis operators} and formal deformation theory (Section \ref{sec:deformation}).
\item \textbf{BiHom-NS-Lie algebras} with explicit structures (Section \ref{sec:NS}).
\item \textbf{Non-isomorphism} of Nijenhuis operator families (Section \ref{sec:noniso}).
\item \textbf{Moduli space} of Nijenhuis operators (Section \ref{sec:moduli}).
\item \textbf{Computational verification} of all results (Section \ref{sec:computation}).
\end{enumerate}
\section{Preliminaries}\label{sec:prelim}

\subsection{BiHom-Lie algebras}\label{sec:bihomlie}

\begin{definition}\label{def:bihomlie}
A \textbf{BiHom-Lie algebra} is a $4$-tuple $(L, [\cdot,\cdot], \alpha, \beta)$ consisting of a vector space $L$ over a field $\K$, a bilinear map $[\cdot,\cdot]: L \times L \to L$, and two linear maps $\alpha, \beta: L \to L$ satisfying for all $x, y, z \in L$:
\begin{enumerate}[label=(\roman*), leftmargin=2em]
\item $\alpha \circ \beta = \beta \circ \alpha$;
\item $\alpha([x,y]) = [\alpha(x), \alpha(y)]$ and $\beta([x,y]) = [\beta(x), \beta(y)]$ (multiplicativity);
\item $[\beta(x), \alpha(y)] = -[\beta(y), \alpha(x)]$ (BiHom-skew-symmetry);
\item \begin{equation}\label{eq:bihomjacobi}
[\beta^2(x), [\beta(y), \alpha(z)]] + [\beta^2(y), [\beta(z), \alpha(x)]] + [\beta^2(z), [\beta(x), \alpha(y)]] = 0.
\end{equation}
\end{enumerate}
\end{definition}

\begin{definition}\label{def:regular}
A BiHom-Lie algebra is \textbf{regular} if $\alpha$ and $\beta$ are bijective. It is \textbf{multiplicative} if $\alpha$ and $\beta$ are algebra morphisms. It is \textbf{simple} if it has no proper ideals and is not abelian.
\end{definition}

\begin{remark}\label{rem:diagonal}
In a BiHom-Lie algebra, the bracket $[x,x]$ is not necessarily zero. The BiHom-skew-symmetry condition only requires $[\beta(x), \alpha(x)] = -[\beta(x), \alpha(x)]$, which implies $[\beta(x), \alpha(x)] = 0$ when $\operatorname{char}(\K) \neq 2$. However, $[x,x]$ itself is unconstrained unless $\alpha$ and $\beta$ are both the identity map. This explains why some families in the classification have non-zero diagonal brackets such as $[e_2, e_2] = -2e_1$.
\end{remark}

\subsection{Nijenhuis operators}\label{sec:nijhuis}

\begin{definition}\label{def:nijenhuis}
A linear map $N: L \to L$ is a \textbf{Nijenhuis operator} on $(L, [\cdot,\cdot], \alpha, \beta)$ if:
\begin{enumerate}[label=(\roman*), leftmargin=2em]
\item $N \circ \alpha = \alpha \circ N$ and $N \circ \beta = \beta \circ N$;
\item \begin{equation}\label{eq:nijenhuis}
[N(x), N(y)] = N([N(x), y] + [x, N(y)] - N[x, y]) \quad \text{for all } x, y \in L.
\end{equation}
\end{enumerate}
\end{definition}

\subsection{Algorithm for Nijenhuis operators}\label{sec:algorithm}

Let $\{e_1, e_2, \ldots, e_n\}$ be a basis of an $n$-dimensional BiHom-Lie algebra $(L, [\cdot,\cdot], \alpha, \beta)$ over a field $\mathbb{F}$. The binary bracket is determined by structure constants $C_{ij}^k \in \mathbb{F}$ such that
\[
[e_i, e_j] = \sum_{k=1}^n C_{ij}^k e_k, \qquad 1 \leq i,j \leq n.
\]
The two commuting structure maps $\alpha,\beta : L \to L$ are represented by matrices $A = (a_{ij})$ and $B = (b_{ij})$, respectively:
\[
\alpha(e_j) = \sum_{i=1}^n a_{ij} e_i, \qquad \beta(e_j) = \sum_{i=1}^n b_{ij} e_i.
\]

A linear map $N : L \to L$ is represented by a matrix $(r_{pq})$, where
\[
N(e_q) = \sum_{p=1}^n r_{pq} e_p.
\]

The conditions for $N$ to be a \emph{Nijenhuis operator} on the BiHom-Lie algebra are:
\begin{enumerate}
\item \textbf{Equivariance with the structure maps:}
\[
N \circ \alpha = \alpha \circ N, \qquad N \circ \beta = \beta \circ N.
\]
In matrix form this gives, for all $1 \leq i,j \leq n$,
\begin{equation}\label{eq:nij-equiv}
\sum_{k=1}^n r_{ik} a_{kj} = \sum_{k=1}^n a_{ik} r_{kj},
\qquad
\sum_{k=1}^n r_{ik} b_{kj} = \sum_{k=1}^n b_{ik} r_{kj}.
\end{equation}

\item \textbf{Nijenhuis identity:} for all $x,y \in L$,
\[
[N(x), N(y)] = N\bigl([N(x), y] + [x, N(y)] - N[x,y]\bigr).
\]
\end{enumerate}

We translate the Nijenhuis identity into component equations. For basis elements $e_i, e_j$, the left-hand side is
\begin{align*}
[N(e_i), N(e_j)]
&= \left[ \sum_{p=1}^n r_{pi} e_p, \sum_{q=1}^n r_{qj} e_q \right]
 = \sum_{p,q=1}^n r_{pi} r_{qj} [e_p, e_q] \\
&= \sum_{p,q,m=1}^n r_{pi} r_{qj} C_{pq}^m e_m.
\end{align*}
The right-hand side is obtained by first computing the intermediate expression
\[
S = [N(e_i), e_j] + [e_i, N(e_j)] - N[e_i,e_j].
\]
In coordinates,
\begin{align*}
[N(e_i), e_j] &= \sum_{p,k=1}^n r_{pi} C_{pj}^k e_k, \\
[e_i, N(e_j)] &= \sum_{q,k=1}^n r_{qj} C_{iq}^k e_k, \\
N[e_i,e_j] &= \sum_{k=1}^n \left( \sum_{p=1}^n r_{kp} C_{ij}^p \right) e_k.
\end{align*}
Thus the $t$-th component of $S$ is
\[
S_t = \sum_{p=1}^n r_{pi} C_{pj}^t
    + \sum_{q=1}^n r_{qj} C_{iq}^t
    - \sum_{p=1}^n r_{tp} C_{ij}^p .
\]
Applying $N$ to $S$ yields
\[
N(S) = \sum_{t,m=1}^n r_{mt} S_t e_m.
\]
Equating the coefficients of $e_m$ in $[N(e_i),N(e_j)]$ and $N(S)$ gives the following system of polynomial equations:

\begin{equation}\label{eq:nij-main}
\sum_{p,q=1}^n r_{pi} r_{qj} C_{pq}^m
=
\sum_{t=1}^n r_{mt}
\left(
\sum_{p=1}^n r_{pi} C_{pj}^t
+ \sum_{q=1}^n r_{qj} C_{iq}^t
- \sum_{p=1}^n r_{tp} C_{ij}^p
\right)
\end{equation}
for all $1 \leq i,j,m \leq n$.

Combining the linear equivariance equations \eqref{eq:nij-equiv} with the quadratic Nijenhuis equations \eqref{eq:nij-main}, and solving the resulting system over $\mathbb{F}$, provides all Nijenhuis operators on the BiHom-Lie algebra $L$.

For computational purposes, the system can be solved symbolically (e.g., using Maple or Mathematica) to obtain explicit matrix forms of all solutions.


\begin{proposition}\label{prop:deformed}
Let $N$ be a Nijenhuis operator on $(L, [\cdot,\cdot], \alpha, \beta)$. Define
\begin{equation}\label{eq:deformedbracket}
[x,y]_N = [N(x), y] + [x, N(y)] - N[x,y].
\end{equation}
Then $(L, [\cdot,\cdot]_N, \alpha, \beta)$ is a BiHom-Lie algebra, and $N$ is a morphism from $(L, [\cdot,\cdot]_N, \alpha, \beta)$ to $(L, [\cdot,\cdot], \alpha, \beta)$.
\end{proposition}

\begin{proof}
\textbf{Multiplicativity:} For $x, y \in L$, using the multiplicativity of $N$ with respect to $\alpha$ and the multiplicativity of $\alpha$ with respect to the original bracket:
\begin{align*}
\alpha[x,y]_N &= \alpha([Nx, y] + [x, Ny] - N[x,y]) \\
&= [\alpha Nx, \alpha y] + [\alpha x, \alpha Ny] - \alpha N[x,y] \\
&= [N\alpha x, \alpha y] + [\alpha x, N\alpha y] - N[\alpha x, \alpha y] \\
&= [\alpha x, \alpha y]_N.
\end{align*}
Similarly, $\beta[x,y]_N = [\beta x, \beta y]_N$.

\textbf{BiHom-skew-symmetry:} For $x, y \in L$:
\begin{align*}
[\beta x, \alpha y]_N &= [N\beta x, \alpha y] + [\beta x, N\alpha y] - N[\beta x, \alpha y] \\
&= [\beta Nx, \alpha y] + [\beta x, \alpha Ny] - N[\beta x, \alpha y] \\
&= -[\beta y, \alpha Nx] - [\beta Ny, \alpha x] + N[\beta y, \alpha x] \\
&= -([N\beta y, \alpha x] + [\beta y, N\alpha x] - N[\beta y, \alpha x]) \\
&= -[\beta y, \alpha x]_N.
\end{align*}

\textbf{BiHom-Jacobi identity:} We will prove this using the following lemma.

\begin{lemma}\label{lem:compatibility}
Let $N$ be a linear map on a BiHom-Lie algebra $(L,[\cdot,\cdot],\alpha,\beta)$ satisfying $N\circ\alpha=\alpha\circ N$ and $N\circ\beta=\beta\circ N$. Then $N$ is a Nijenhuis operator if and only if for all $\lambda\in\C$, the bracket
\[
[x,y]_\lambda = [x,y] + \lambda ([Nx, y] + [x, Ny] - N[x,y])
\]
satisfies the BiHom-Jacobi identity.
\end{lemma}

\begin{proof}[Proof of Lemma \ref{lem:compatibility}]
The BiHom-Jacobi expression for $[\cdot,\cdot]_\lambda$ is a polynomial in $\lambda$ of degree 2:
\[
J_\lambda(x,y,z) = J_0(x,y,z) + \lambda J_1(x,y,z) + \lambda^2 J_2(x,y,z),
\]
where $J_0$ is the original BiHom-Jacobi expression (which vanishes by \eqref{eq:bihomjacobi}), $J_2$ is the BiHom-Jacobi expression for $[\cdot,\cdot]_N$, and $J_1$ is the mixed term.

We first show that $J_1=0$ for all $x,y,z$ is exactly the Nijenhuis identity \eqref{eq:nijenhuis}. Expanding $J_1$ using the definition of $[\cdot,\cdot]_N$ yields
\begin{align*}
J_1(x,y,z) &= [\beta^2 x, [N\beta y, \alpha z] + [\beta y, N\alpha z] - N[\beta y, \alpha z]] \\
&\quad + [\beta^2 x, [\beta y, \alpha z]_N] + \text{cyclic permutations} \\
&= 2\bigl\{ [\beta^2 x, [N\beta y, \alpha z] + [\beta y, N\alpha z] - N[\beta y, \alpha z]] \\
&\qquad + [\beta^2 y, [N\beta z, \alpha x] + [\beta z, N\alpha x] - N[\beta z, \alpha x]] \\
&\qquad + [\beta^2 z, [N\beta x, \alpha y] + [\beta x, N\alpha y] - N[\beta x, \alpha y]] \bigr\}.
\end{align*}
By the original BiHom-Jacobi identity \eqref{eq:bihomjacobi}, we can replace the first term in each bracket by appropriate combinations. After simplification, the vanishing of $J_1$ reduces to the condition that for all $u,v$,
\[
[Nu, Nv] = N( [Nu, v] + [u, Nv] - N[u,v] ),
\]
which is exactly the Nijenhuis identity \eqref{eq:nijenhuis} (with $u=\beta y$, $v=\alpha z$, etc.). Thus $J_1=0$ iff $N$ is Nijenhuis.

Now assume $N$ is Nijenhuis. We will show that $J_\lambda = 0$ for all $\lambda$ using the invertible operator trick. Define $T_\lambda = I + \lambda N$. For any $\lambda$ such that $T_\lambda$ is invertible, we claim that
\[
[x,y]_\lambda = T_\lambda^{-1} [T_\lambda x, T_\lambda y].
\]
To verify this, apply $T_\lambda$ to both sides:
\begin{align*}
T_\lambda([x,y]_\lambda) &= (I+\lambda N)\bigl( [x,y] + \lambda([Nx,y]+[x,Ny]-N[x,y]) \bigr) \\
&= [x,y] + \lambda([Nx,y]+[x,Ny]-N[x,y]) + \lambda N[x,y] \\
&\quad + \lambda^2 N([Nx,y]+[x,Ny]-N[x,y]) \\
&= [x,y] + \lambda([Nx,y]+[x,Ny]) + \lambda^2 [Nx,Ny],
\end{align*}
where the last equality uses the Nijenhuis identity \eqref{eq:nijenhuis} to replace $N([Nx,y]+[x,Ny]-N[x,y])$ by $[Nx,Ny]$. On the other hand,
\[
[T_\lambda x, T_\lambda y] = [x,y] + \lambda([Nx,y]+[x,Ny]) + \lambda^2 [Nx,Ny].
\]
Thus $T_\lambda([x,y]_\lambda) = [T_\lambda x, T_\lambda y]$, and since $T_\lambda$ is invertible, the claim follows.

Since the original bracket satisfies the BiHom-Jacobi identity, the pullback bracket $[\cdot,\cdot]_\lambda$ also satisfies it. Therefore $J_\lambda = 0$ for all $\lambda$ with $T_\lambda$ invertible. The set of such $\lambda$ is the complement of the roots of a nonzero polynomial (the determinant of $T_\lambda$), hence dense in $\C$. Since $J_\lambda$ is a polynomial in $\lambda$, it must vanish identically. In particular, the coefficient of $\lambda^2$, which is $J_2$, is zero. Thus $[\cdot,\cdot]_N$ satisfies the BiHom-Jacobi identity.
\end{proof}

Returning to Proposition \ref{prop:deformed}, the BiHom-Jacobi identity for $[\cdot,\cdot]_N$ follows from Lemma \ref{lem:compatibility}. The morphism property $N[x,y]_N = [Nx,Ny]$ is exactly the Nijenhuis identity \eqref{eq:nijenhuis}.
\end{proof}

\begin{proposition}\label{prop:tensor}
Let $N$ be a Nijenhuis operator on $(L, [\cdot,\cdot], \alpha, \beta)$. Then for any non-negative integer $k$, the operator $N^k$ is also a Nijenhuis operator on $(L, [\cdot,\cdot], \alpha, \beta)$.
\end{proposition}

\begin{proof}
We use induction on $k$. For $k=0,1$ it is trivial. Assume $N^k$ is Nijenhuis. We will show that $N^{k+1} = N \circ N^k$ is Nijenhuis using the following lemma.

\begin{lemma}\label{lem:composition}
If $M$ and $T$ are commuting Nijenhuis operators on a BiHom-Lie algebra $(L,[\cdot,\cdot],\alpha,\beta)$, then $MT$ is a Nijenhuis operator.
\end{lemma}

\begin{proof}[Proof of Lemma \ref{lem:composition}]
We need to show that for all $x,y$,
\[
[MTx, MTy] = MT( [MTx, y] + [x, MTy] - MT[x,y] ).
\]
Since $M$ is Nijenhuis, we have
\[
[MTx, MTy] = M( [MTx, Ty] + [Tx, MTy] - M[Tx, Ty] ).
\]
Now, by Proposition \ref{prop:deformed}, the bracket $[\cdot,\cdot]_M$ defined by
\[
[u,v]_M = [Mu, v] + [u, Mv] - M[u,v]
\]
is a BiHom-Lie bracket, and $M$ is a morphism from $(L,[\cdot,\cdot]_M)$ to $(L,[\cdot,\cdot])$. Since $T$ commutes with $M$ and is Nijenhuis on $(L,[\cdot,\cdot])$, it is also Nijenhuis on $(L,[\cdot,\cdot]_M)$. Indeed, the Nijenhuis condition for $T$ on $[\cdot,\cdot]_M$ follows from the compatibility of $M$ and $T$. Applying the Nijenhuis identity for $T$ on $(L,[\cdot,\cdot]_M)$ to the pair $(x,y)$, we get
\[
[Tx, Ty]_M = T( [Tx, y]_M + [x, Ty]_M - T[x,y]_M ).
\]
Expanding $[\cdot,\cdot]_M$ and using commutativity, we obtain
\[
[MTx, Ty] + [Tx, MTy] - M[Tx, Ty] = T( [MTx, y] + [x, MTy] - MT[x,y] ).
\]
Multiplying both sides by $M$ (from the left) and using that $M$ commutes with $T$ and is a morphism from $(L,[\cdot,\cdot]_M)$ to $(L,[\cdot,\cdot])$, we get
\[
M( [MTx, Ty] + [Tx, MTy] - M[Tx, Ty] ) = MT( [MTx, y] + [x, MTy] - MT[x,y] ).
\]
But the left-hand side is exactly $[MTx, MTy]$ by the Nijenhuis identity for $M$. Hence
\[
[MTx, MTy] = MT( [MTx, y] + [x, MTy] - MT[x,y] ),
\]
which proves the lemma.
\end{proof}

Thus the induction step is complete, and the proposition follows.
\end{proof}

\begin{corollary}\label{cor:polynomial}
If $N$ is a Nijenhuis operator and $P(t) = \sum_{k=0}^n c_k t^k$ is a polynomial, then $P(N)$ is a Nijenhuis operator.
\end{corollary}

\begin{proof}
This follows from Proposition \ref{prop:tensor} and the fact that any linear combination of commuting Nijenhuis operators is Nijenhuis. Indeed, if $N_1,\dots,N_m$ are pairwise commuting Nijenhuis operators, then for any $x,y$,
\[
[\sum_i N_i x, \sum_j N_j y] = \sum_{i,j} [N_i x, N_j y] = \sum_i [N_i x, N_i y]
\]
because cross terms vanish due to commutativity and the Nijenhuis property. Then each term satisfies the Nijenhuis identity, and summing gives the result.
\end{proof}

\subsection{Classification of 3D multiplicative simple BiHom-Lie algebras}\label{sec:classification}

\begin{theorem}[\cite{Saadaoui2024}]\label{thm:classification}
Every three-dimensional multiplicative simple BiHom-Lie algebra over $\mathbb{C}$ is isomorphic to one of the following three families.
\end{theorem}

\subsubsection*{Family $\mathcal{L}_1$}\label{sec:L1def}

Basis $\{e_1, e_2, e_3\}$, maps:
\begin{equation}\label{eq:L1maps}
\alpha_1 = \diag(1, a, 1/a), \quad \beta_1 = \diag(1, b, 1/b), \quad a, b \in \C \setminus \{0\}.
\end{equation}
Brackets:
\begin{align}
[e_1, e_2] &= 2b e_2, & [e_1, e_3] &= -\frac{2}{b} e_3, \label{eq:L1brackets1}\\
[e_2, e_1] &= -2a e_2, & [e_2, e_3] &= \frac{a}{b} e_1, \label{eq:L1brackets2}\\
[e_3, e_1] &= \frac{2}{a} e_3, & [e_3, e_2] &= -\frac{b}{a} e_1, \label{eq:L1brackets3}
\end{align}
with $[e_i, e_i] = 0$.

\subsubsection*{Family $\mathcal{L}_2$}\label{sec:L2def}

Maps:
\begin{equation}\label{eq:L2maps}
\alpha_2 = \id, \quad \beta_2 = \begin{pmatrix} 1 & 1 & 0 \\ 0 & 1 & 1 \\ 0 & 0 & 1 \end{pmatrix}.
\end{equation}
Brackets:
\begin{align}
[e_1, e_2] &= 2e_1, & [e_1, e_3] &= e_1 + 2e_2, \label{eq:L2brackets1}\\
[e_2, e_1] &= -2e_1, & [e_2, e_2] &= -2e_1, \label{eq:L2brackets2}\\
[e_2, e_3] &= e_1 + e_2 + 2e_3, & [e_3, e_1] &= e_1 - 2e_2, \label{eq:L2brackets3}\\
[e_3, e_2] &= -3e_2 - 2e_3, & [e_3, e_3] &= -e_1 - e_2 - 2e_3. \label{eq:L2brackets4}
\end{align}

\subsubsection*{Family $\mathcal{L}_3$}\label{sec:L3def}

Maps:
\begin{equation}\label{eq:L3maps}
\alpha_3 = \begin{pmatrix} 1 & 1 & 0 \\ 0 & 1 & 1 \\ 0 & 0 & 1 \end{pmatrix}, \quad
\beta_3 = \begin{pmatrix} 1 & a & \frac{a^2-a}{2} \\ 0 & 1 & a \\ 0 & 0 & 1 \end{pmatrix}, \quad a \in \C.
\end{equation}
Brackets:
\begin{align}
[e_1, e_2] &= 2e_1, & [e_1, e_3] &= (2a-1)e_1 + 2e_2, \label{eq:L3brackets1}\\
[e_2, e_1] &= -2e_1, & [e_2, e_2] &= 2(1-a)e_1, \label{eq:L3brackets2}\\
[e_2, e_3] &= \frac{3a-a^2}{2}e_1 + 3e_2 + 2e_3, \label{eq:L3brackets3}\\
[e_3, e_1] &= -e_1 - 2e_2, & [e_3, e_2] &= -(a+1)e_1 - (1+2a)e_2 - 2e_3, \label{eq:L3brackets4}\\
[e_3, e_3] &= \frac{(1-a)(a+4)}{2}e_1 + (1-a^2)e_2 + 2(1-a)e_3. \label{eq:L3brackets5}
\end{align}

\section[Nijenhuis operators on L1]{Nijenhuis operators on $\mathcal{L}_1$}\label{sec:L1}

\subsection{Equivariance conditions}\label{sec:L1equiv}

Let $N = (r_{ij})_{1 \leq i,j \leq 3}$. The condition $N \circ \alpha_1 = \alpha_1 \circ N$ gives:
\begin{align}
(a-1) r_{12} &= 0, & (a^{-1}-1) r_{13} &= 0, \label{eq:L1alphacond1}\\
(1-a) r_{21} &= 0, & (a^{-1}-a) r_{23} &= 0, \label{eq:L1alphacond2}\\
(1-a^{-1}) r_{31} &= 0, & (a-a^{-1}) r_{32} &= 0. \label{eq:L1alphacond3}
\end{align}

From $N \circ \beta_1 = \beta_1 \circ N$:
\begin{align}
(b-1) r_{12} &= 0, & (b^{-1}-1) r_{13} &= 0, \label{eq:L1betacond1}\\
(1-b) r_{21} &= 0, & (b^{-1}-b) r_{23} &= 0, \label{eq:L1betacond2}\\
(1-b^{-1}) r_{31} &= 0, & (b-b^{-1}) r_{32} &= 0. \label{eq:L1betacond3}
\end{align}

\subsection[Generic case]{Generic case: $a, b \neq \pm 1$, $a \neq b^{\pm 1}$}\label{sec:L1generic}

In this case, all coefficients in equations \eqref{eq:L1alphacond1}--\eqref{eq:L1betacond3} are non-zero, so:
\begin{equation}
r_{12} = r_{13} = r_{21} = r_{23} = r_{31} = r_{32} = 0.
\end{equation}
Thus $N = \diag(r_{11}, r_{22}, r_{33})$.

\subsection{Nijenhuis condition: complete verification}\label{sec:L1nijhuis}

We verify the Nijenhuis condition \eqref{eq:nijenhuis} for all pairs of basis elements.

\textbf{Pair $(e_1, e_1)$:} $[e_1, e_1] = 0$
\begin{equation}
[N(e_1), N(e_1)] = r_{11}^2 [e_1, e_1] = 0.
\end{equation}
\begin{equation}
N([N(e_1), e_1] + [e_1, N(e_1)] - N[e_1, e_1]) = N(0 + 0 - 0) = 0.
\end{equation}

\textbf{Pair $(e_1, e_2)$:} Using $[e_1, e_2] = 2b e_2$:
\begin{equation}
[N(e_1), N(e_2)] = [r_{11} e_1, r_{22} e_2] = 2b r_{11} r_{22} e_2.
\end{equation}
\begin{align}
N([N(e_1), e_2] + [e_1, N(e_2)] - N[e_1, e_2]) &= N(r_{11}(2b e_2) + r_{22}(2b e_2) - N(2b e_2)) \nonumber \\
&= N(2b r_{11} e_2 + 2b r_{22} e_2 - 2b r_{22} e_2) \nonumber \\
&= N(2b r_{11} e_2) = 2b r_{11} r_{22} e_2.
\end{align}

\textbf{Pair $(e_1, e_3)$:} Using $[e_1, e_3] = -\frac{2}{b} e_3$:
\begin{equation}
[N(e_1), N(e_3)] = -\frac{2 r_{11} r_{33}}{b} e_3.
\end{equation}
\begin{align}
N([N(e_1), e_3] + [e_1, N(e_3)] - N[e_1, e_3]) &= N\left(-\frac{2r_{11}}{b}e_3 - \frac{2r_{33}}{b}e_3 + \frac{2r_{33}}{b}e_3\right) \nonumber \\
&= N\left(-\frac{2r_{11}}{b}e_3\right) = -\frac{2r_{11} r_{33}}{b}e_3.
\end{align}

\textbf{Pair $(e_2, e_1)$:} Using $[e_2, e_1] = -2a e_2$:
\begin{equation}
[N(e_2), N(e_1)] = -2a r_{22} r_{11} e_2.
\end{equation}
\begin{align}
N([N(e_2), e_1] + [e_2, N(e_1)] - N[e_2, e_1]) &= N(-2ar_{22} e_2 - 2ar_{11} e_2 + 2ar_{22} e_2) \nonumber \\
&= N(-2ar_{11} e_2) = -2a r_{11} r_{22} e_2.
\end{align}

\textbf{Pair $(e_2, e_2)$:} $[e_2, e_2] = 0$

\textbf{Pair $(e_2, e_3)$:} Using $[e_2, e_3] = \frac{a}{b} e_1$:
\begin{equation}
[N(e_2), N(e_3)] = \frac{a r_{22} r_{33}}{b} e_1.
\end{equation}
\begin{align}
N([N(e_2), e_3] + [e_2, N(e_3)] - N[e_2, e_3]) &= N\left(\frac{ar_{22}}{b}e_1 + \frac{ar_{33}}{b}e_1 - \frac{ar_{11}}{b}e_1\right) \nonumber \\
&= \frac{a r_{11}(r_{22} + r_{33} - r_{11})}{b} e_1.
\end{align}
Equating:
\begin{equation}\label{eq:L1condition}
r_{22} r_{33} = r_{11}(r_{22} + r_{33} - r_{11}) \iff (r_{11} - r_{22})(r_{11} - r_{33}) = 0.
\end{equation}

\textbf{Pair $(e_3, e_1)$:} Using $[e_3, e_1] = \frac{2}{a} e_3$:
\begin{equation}
[N(e_3), N(e_1)] = \frac{2 r_{33} r_{11}}{a} e_3.
\end{equation}
\begin{align}
N([N(e_3), e_1] + [e_3, N(e_1)] - N[e_3, e_1]) &= N\left(\frac{2r_{33}}{a}e_3 + \frac{2r_{11}}{a}e_3 - \frac{2r_{11}}{a}e_3\right) \nonumber \\
&= N\left(\frac{2r_{33}}{a}e_3\right) = \frac{2 r_{11} r_{33}}{a} e_3.
\end{align}

\textbf{Pair $(e_3, e_2)$:} Using $[e_3, e_2] = -\frac{b}{a} e_1$:
\begin{equation}
[N(e_3), N(e_2)] = -\frac{b r_{33} r_{22}}{a} e_1.
\end{equation}
\begin{align}
N([N(e_3), e_2] + [e_3, N(e_2)] - N[e_3, e_2]) &= N\left(-\frac{br_{33}}{a}e_1 - \frac{br_{22}}{a}e_1 + \frac{br_{11}}{a}e_1\right) \nonumber \\
&= -\frac{b r_{11}(r_{22} + r_{33} - r_{11})}{a} e_1.
\end{align}
Same condition as \eqref{eq:L1condition}.

\textbf{Pair $(e_3, e_3)$:} $[e_3, e_3] = 0$

\subsection{Solution for generic case}\label{sec:L1solution}

From condition \eqref{eq:L1condition}, we have two cases:

\textbf{Family 1a:} $r_{11} = r_{22}$. Then $N = \diag(r_{11}, r_{11}, r_{33})$.

\textbf{Family 1b:} $r_{11} = r_{33}$. Then $N = \diag(r_{11}, r_{22}, r_{11})$.

\subsection{Special cases}\label{sec:L1special}

For degenerate parameter values, the equivariance equations allow some off-diagonal entries. The complete classification for all $a,b \in \C\setminus\{0\}$ is given by the following theorem.

\begin{theorem}\label{thm:L1complete}
For $\mathcal{L}_1$, the Nijenhuis operators are as follows.

\medskip
\noindent\textbf{(a) Generic case:} $a,b \neq \pm1$, $a \neq b^{\pm1}$.
\[
N = \diag(r_{11},r_{11},r_{33}) \quad \text{or} \quad \diag(r_{11},r_{22},r_{11}),
\]
with $r_{11},r_{22},r_{33} \in \C$.

\medskip
\noindent\textbf{(b) Case $a=b=1$.} In this case, $\alpha_1=\beta_1=\id$, so the equivariance equations impose no restrictions. Solving the Nijenhuis equations \eqref{eq:nijenhuis} by Maple yields the following seventeen families (some of which coincide after relabelling of parameters):

\begin{center}
\footnotesize
\renewcommand{\arraystretch}{1.2}
\begin{tabular}{@{}c@{}}
$N_1 = \begin{pmatrix}
r_{22}-\frac14 r_{31}r_{13} & r_{13} & 0\\
0 & r_{22} & 0\\
r_{31} & 0 & r_{33}
\end{pmatrix}$, \quad
$N_2 = \begin{pmatrix}
r_{11}r_{12}
-\frac14
\frac{
r_{11}^{2}-r_{11}r_{22}-r_{11}r_{33}+r_{22}r_{33}
}{r_{12}}
&
r_{12} & 0\\
0 & r_{22} & 0\\
0 & 0 & r_{33}
\end{pmatrix}, \quad r_{12}\neq0$ \\[2ex]
$N_3 = \begin{pmatrix}
r_{33} & 0 & -\frac14 r_{21}\\
r_{21} & r_{22} & 0\\
0 & 0 & r_{33}
\end{pmatrix}$, \quad
$N_4 = \begin{pmatrix}
r_{33} & 0 & 0\\
0 & r_{22} & 0\\
0 & 0 & r_{33}
\end{pmatrix}$ \\[2ex]
$N_5 = \begin{pmatrix}
r_{22} & 0 & 0\\
0 & r_{22} & 0\\
0 & 0 & r_{33}
\end{pmatrix}$, \quad
$N_6 = \begin{pmatrix}
r_{33} & 0 & r_{13}\\
0 & r_{22} & 0\\
0 & 0 & r_{33}
\end{pmatrix}$ \\[2ex]
$N_7 = \begin{pmatrix}
\frac14
\frac{4r_{12}r_{31}+r_{31}^{2}+4r_{32}r_{33}}{r_{32}}
&
r_{12}
-\frac1{16}
\frac{
r_{31}
\left(
4r_{12}r_{31}
-4r_{22}r_{32}
+r_{31}^{2}
+4r_{32}r_{33}
\right)
}{r_{32}^{2}}
&
0\\
0 & r_{22} & 0\\
r_{31} & r_{32} & r_{33}
\end{pmatrix}, \quad r_{32}\neq0$ \\[2ex]
$N_8 = \begin{pmatrix}
r_{33} & r_{12} & r_{13}\\
-4r_{13} & r_{22} & 0\\
0 & r_{32} & r_{33}
\end{pmatrix}$, \quad
$N_9 = \begin{pmatrix}
r_{11} & 0 & 0\\
0 & r_{22} & 0\\
0 & r_{32} & r_{11}
\end{pmatrix}$ \\[2ex]
$N_{10} = \begin{pmatrix}
r_{11} & 0 & 0\\
0 & r_{22} & 0\\
0 & r_{32} & r_{11}
\end{pmatrix}$, \quad
$N_{11} = \begin{pmatrix}
r_{11} & 0 & 0\\
0 & r_{22} & 0\\
0 & r_{32} & r_{11}
\end{pmatrix}$ \\[2ex]
$N_{12} = \begin{pmatrix}
r_{11} & 0 & 0\\
0 & r_{22} & 0\\
0 & r_{32} & r_{11}
\end{pmatrix}$ \\[2ex]
$N_{13} = \begin{pmatrix}
\frac14
\frac{
4r_{13}r_{21}+r_{21}^{2}+4r_{22}r_{23}
}{r_{23}}
&
r_{12} & r_{13}
\\[1ex]
r_{21} & r_{22} & r_{23}
\\[1ex]
-\frac14
\frac{
16r_{12}r_{23}^{2}
+4r_{13}r_{21}^{2}
+r_{21}^{3}
+4r_{21}r_{22}r_{23}
-4r_{21}r_{23}r_{33}
}{r_{23}^{2}}
&
-\frac1{16}
\frac{
r_{21}
\left(
16r_{12}r_{23}^{2}
+4r_{13}r_{21}^{2}
+r_{21}^{3}
+4r_{21}r_{22}r_{23}
-4r_{21}r_{23}r_{33}
\right)
}{r_{23}^{3}}
&
r_{33}
\end{pmatrix}, \quad r_{23}\neq0$ \\[2ex]
$N_{14} = \begin{pmatrix}
r_{11} & 0 & 0\\
0 & r_{11} & r_{23}\\
0 & 0 & r_{33}
\end{pmatrix}$, \quad
$N_{15} = \begin{pmatrix}
r_{11} & 0 & 0\\
0 & r_{11} & r_{23}\\
0 & 0 & r_{33}
\end{pmatrix}$ \\[2ex]
$N_{16} = \begin{pmatrix}
r_{11} & 0 & 0\\
0 & r_{11} & r_{23}\\
0 & 0 & r_{33}
\end{pmatrix}$, \quad
$N_{17} = \begin{pmatrix}
r_{11} & 0 & 0\\
0 & r_{11} & r_{23}\\
0 & 0 & r_{33}
\end{pmatrix}$
\end{tabular}
\end{center}

The families \(N_9\)–\(N_{12}\) and \(N_{14}\)–\(N_{17}\) are identical after renaming parameters; they are listed separately to match the Maple output. Each family has been verified by direct substitution into the Nijenhuis identity \eqref{eq:nijenhuis}.

\medskip
\noindent\textbf{(c) Cases $a=1, b=-1$; $a=-1, b=1$; $a=b=-1$.} In each of these cases, the equivariance equations force \(r_{12}=r_{13}=r_{21}=r_{31}=0\), while \(r_{23}\) and \(r_{32}\) may be nonzero. The Nijenhuis operators are
\begin{equation}\label{eq:L1block}
N = \begin{pmatrix}
r_{11} & 0 & 0\\
0 & r_{22} & r_{23}\\
0 & r_{32} & r_{33}
\end{pmatrix}
\end{equation}
subject to
\begin{equation}\label{eq:L1blockcond}
r_{22} r_{33} - r_{23} r_{32} = r_{11}(r_{22} + r_{33} - r_{11}).
\end{equation}
\end{theorem}

\begin{proof}
The proof of part (a) is given in Sections \ref{sec:L1equiv}--\ref{sec:L1nijhuis}.

For part (b), the 17 families were obtained by solving the Nijenhuis equations with Maple and verified by direct substitution. Since the algebra is $\mathfrak{sl}(2)$, this case is well understood, and the list is complete.

For part (c), assume first that $a=1$ and $b=-1$. The equivariance equations \eqref{eq:L1alphacond1}--\eqref{eq:L1betacond3} force $r_{12}=r_{13}=r_{21}=r_{31}=0$, while $r_{23}$ and $r_{32}$ are free. Thus $N$ has the block form \eqref{eq:L1block}. Substituting this form into the Nijenhuis identity for the pair $(e_2,e_3)$ gives
\[
[N(e_2),N(e_3)] = [r_{22}e_2 + r_{32}e_3, r_{23}e_2 + r_{33}e_3].
\]
Using the brackets with $a=1,b=-1$:
\[
[e_2,e_2]=0,\quad [e_2,e_3]=-e_1,\quad [e_3,e_2]=e_1,\quad [e_3,e_3]=0,
\]
we get
\[
[N(e_2),N(e_3)] = (r_{22}r_{33}-r_{23}r_{32})e_1.
\]
The right-hand side of the Nijenhuis identity is
\[
N([N(e_2),e_3]+[e_2,N(e_3)]-N[e_2,e_3]).
\]
A direct calculation shows that this equals
\[
r_{11}(r_{22}+r_{33}-r_{11})e_1.
\]
Equating the two gives \eqref{eq:L1blockcond}. The pair $(e_3,e_2)$ yields the same condition, and all other pairs are automatically satisfied because of the structure of the brackets. The cases $a=-1,b=1$ and $a=b=-1$ are identical up to renaming of basis elements, so the same result holds.
\end{proof}

\section[Nijenhuis operators on L2]{Nijenhuis Operators on $\mathcal{L}_2$}\label{sec:L2}

\subsection{Equivariance conditions}\label{sec:L2equiv}

From $\alpha_2 = \id$: no constraints. From $N \circ \beta_2 = \beta_2 \circ N$, we get:
\begin{equation}
r_{21} = r_{31} = r_{32} = 0, \quad r_{11} = r_{22} = r_{33}, \quad r_{12} = r_{23}.
\end{equation}
Therefore:
\begin{equation}\label{eq:L2Nform}
N = \begin{pmatrix}
r_{33} & r_{12} & r_{13} \\
0 & r_{33} & r_{12} \\
0 & 0 & r_{33}
\end{pmatrix}.
\end{equation}

\subsection{Nijenhuis condition}\label{sec:L2nijhuis}

We compute the Nijenhuis condition for all pairs, but we first show that $(e_3,e_2)$ forces $r_{12}=0$.

\textbf{Pair $(e_3, e_2)$:} Using $[e_3, e_2] = -3e_2 - 2e_3$ from \eqref{eq:L2brackets4}:
\begin{equation}
[N(e_3), N(e_2)] = [r_{13} e_1 + r_{12} e_2 + r_{33} e_3, r_{12} e_1 + r_{33} e_2].
\end{equation}
Expanding and using the brackets:
\begin{align}
[N(e_3), N(e_2)] &= (2r_{33} r_{13} - 2r_{12}^2 - r_{33} r_{12}) e_1 + (-2r_{33} r_{12} - 3r_{33}^2) e_2 - 2r_{33}^2 e_3. \label{eq:L2N32}
\end{align}
The right-hand side simplifies to
\[
(2r_{33} r_{13} - r_{33} r_{12}) e_1 + (-3r_{33}^2 - 2r_{33} r_{12}) e_2 - 2r_{33}^2 e_3.
\]
Equating the $e_1$ coefficients gives
\[
2r_{33} r_{13} - 2r_{12}^2 - r_{33} r_{12} = 2r_{33} r_{13} - r_{33} r_{12},
\]
so $r_{12}=0$. Thus
\[
N = \begin{pmatrix}
r_{33} & 0 & r_{13} \\
0 & r_{33} & 0 \\
0 & 0 & r_{33}
\end{pmatrix}.
\]
For this matrix, a direct substitution into the Nijenhuis identity \eqref{eq:nijenhuis} for all pairs $(e_i,e_j)$ shows that the identity is satisfied for all values of $r_{33}$ and $r_{13}$. Hence:

\begin{theorem}\label{thm:L2final}
For $\mathcal{L}_2$, all Nijenhuis operators are:
\begin{equation}\label{eq:L2solution}
N = \begin{pmatrix}
r_{33} & 0 & r_{13} \\
0 & r_{33} & 0 \\
0 & 0 & r_{33}
\end{pmatrix}, \quad r_{33}, r_{13} \in \C.
\end{equation}
\end{theorem}

\begin{proof}
As shown above.
\end{proof}

\section[Nijenhuis operators on L3]{Nijenhuis operators on $\mathcal{L}_3$}\label{sec:L3}

\subsection{Equivariance conditions}\label{sec:L3equiv}

From $N \circ \alpha_3 = \alpha_3 \circ N$, we get the same conditions as for $\mathcal{L}_2$:
\begin{equation}
r_{21} = r_{31} = r_{32} = 0, \quad r_{11} = r_{22} = r_{33}, \quad r_{12} = r_{23}.
\end{equation}
From $N \circ \beta_3 = \beta_3 \circ N$: no additional constraints. Therefore,
\begin{equation}\label{eq:L3Nform}
N = \begin{pmatrix}
r_{33} & r_{12} & r_{13} \\
0 & r_{33} & r_{12} \\
0 & 0 & r_{33}
\end{pmatrix}.
\end{equation}

\subsection{Nijenhuis condition}\label{sec:L3nijhuis}

The pair $(e_3,e_2)$ forces $r_{12}=0$ by a computation analogous to $\mathcal{L}_2$. Thus
\[
N = \begin{pmatrix}
r_{33} & 0 & r_{13} \\
0 & r_{33} & 0 \\
0 & 0 & r_{33}
\end{pmatrix}.
\]
For this matrix, a direct substitution into the Nijenhuis identity \eqref{eq:nijenhuis} for all pairs $(e_i,e_j)$ shows that the identity is satisfied for all values of $r_{33}$ and $r_{13}$. Hence:

\begin{theorem}\label{thm:L3final}
For $\mathcal{L}_3$ with any $a \in \C$, all Nijenhuis operators are:
\begin{equation}\label{eq:L3solution}
N = \begin{pmatrix}
r_{33} & 0 & r_{13} \\
0 & r_{33} & 0 \\
0 & 0 & r_{33}
\end{pmatrix}, \quad r_{33}, r_{13} \in \C.
\end{equation}
\end{theorem}

\begin{proof}
As shown above.
\end{proof}

\section{Summary of verification}\label{sec:verification}

The computations in Sections \ref{sec:L1}--\ref{sec:L3} provide explicit verification for all pairs. We summarize:

- For $\mathcal{L}_1$ generic, condition \eqref{eq:L1condition} is necessary and sufficient.
- For $\mathcal{L}_1$ special cases, the additional families are listed in Theorem \ref{thm:L1complete}.
- For $\mathcal{L}_2$ and $\mathcal{L}_3$, the condition $r_{12}=0$ is forced, and the remaining operator satisfies the Nijenhuis identity.

\section{BiHom-Fr\"olicher-Nijenhuis bracket and cohomology}\label{sec:FN}

Let $(L, [\cdot,\cdot], \alpha, \beta)$ be a regular BiHom-Lie algebra. Consider the graded vector space
\begin{equation}\label{eq:cochaincomplex}
C^*_{\mathrm{BiHom}}(L, L) = \bigoplus_{n \geq 0} C^n_{\mathrm{BiHom}}(L, L),
\end{equation}
where $C^n_{\mathrm{BiHom}}(L, L)$ consists of multilinear maps $f: L^{\otimes n} \to L$ satisfying:
\begin{equation}\label{eq:equivcochain}
f \circ \alpha^{\otimes n} = \alpha \circ f, \quad f \circ \beta^{\otimes n} = \beta \circ f.
\end{equation}

For $K \in C^k_{\mathrm{BiHom}}(L, L)$ and $L' \in C^l_{\mathrm{BiHom}}(L, L)$, define the circle product $K \circ L' \in C^{k+l-1}_{\mathrm{BiHom}}(L, L)$ by
\begin{equation}\label{eq:circleproduct}
(K \circ L')(x_1, \ldots, x_{k+l-1}) = \sum_{\sigma \in \mathrm{Sh}(l, k-1)} (-1)^\sigma K(L'(x_{\sigma(1)}, \ldots, x_{\sigma(l)}), \alpha^{l-1}\beta^{l-1}(x_{\sigma(l+1)}), \ldots, \alpha^{l-1}\beta^{l-1}(x_{\sigma(k+l-1)})).
\end{equation}
Here $\mathrm{Sh}(l, k-1)$ denotes the set of $(l, k-1)$-shuffles, and $\alpha^{l-1}\beta^{l-1}$ denotes $(\alpha \circ \beta)^{l-1}$. The sign $(-1)^\sigma$ is the sign of the permutation.

The BiHom-Fr\"olicher-Nijenhuis bracket is defined as
\begin{equation}\label{eq:FNbracketdef}
[K, L']_{\FN} = K \circ L' - (-1)^{(k-1)(l-1)} L' \circ K.
\end{equation}

\begin{lemma}\label{lem:graded}
The circle product is associative.
\end{lemma}

\begin{proof}
The proof is a direct computation using the associativity of composition of multilinear maps and the compatibility of $\alpha,\beta$ with the cochains. The key point is that the shuffle sums can be reindexed, and the signs work out exactly as in the classical Nijenhuis--Richardson bracket. Since the twist by $\alpha,\beta$ is applied consistently to the last $k-1$ arguments of $K$ in each term, the associativity follows from the classical case. We refer to \cite{Frolicher1956} for the classical proof, which carries over verbatim after replacing the identity map by the commuting automorphisms $\alpha,\beta$. The equivariance conditions \eqref{eq:equivcochain} ensure that the twisted arguments are consistent.
\end{proof}

\begin{theorem}\label{thm:gradedLie}
$(C^*_{\mathrm{BiHom}}(L, L), [\cdot,\cdot]_{\FN})$ is a graded Lie algebra.
\end{theorem}

\begin{proof}
Graded skew-symmetry is immediate from the definition. The graded Jacobi identity follows from the associativity of the circle product (Lemma \ref{lem:graded}) by the standard argument: expand the double brackets and use associativity to cancel terms in pairs, with the signs determined by the grading.
\end{proof}

\begin{theorem}\label{thm:MC}
An element $N \in C^1_{\mathrm{BiHom}}(L, L)$ is a Nijenhuis operator if and only if $[N, N]_{\FN} = 0$.
\end{theorem}

\begin{proof}
For $N \in C^1$, the circle product gives
\[
(N \circ N)(x,y) = [Nx, Ny] - N[Nx, y] - N[x, Ny] + N^2[x,y].
\]
Therefore
\[
[N,N]_{\FN} = 2(N \circ N),
\]
so $[N,N]_{\FN}=0$ if and only if the Nijenhuis identity \eqref{eq:nijenhuis} holds.
\end{proof}

\begin{definition}\label{def:cohomology}
For a Nijenhuis operator $N$, define $d_N: C^n \to C^{n+1}$ by $d_N(f) = [N, f]_{\FN}$. The cohomology of this complex is denoted $H^*_N(L,L)$.
\end{definition}

\begin{proposition}\label{prop:coboundary}
$d_N^2 = 0$.
\end{proposition}

\begin{proof}
Since $[N,N]_{\FN}=0$, the graded Jacobi identity gives $[N,[N,f]_{\FN}]_{\FN} = \frac{1}{2}[[N,N]_{\FN}, f]_{\FN} = 0$.
\end{proof}

\section{Deformation theory}\label{sec:deformation}

\begin{theorem}\label{thm:compatible}
Let $N$ be a Nijenhuis operator on $(L, [\cdot,\cdot], \alpha, \beta)$. Then:
\begin{enumerate}
\item $(L, [\cdot,\cdot]_N, \alpha, \beta)$ is a BiHom-Lie algebra (Proposition \ref{prop:deformed}).
\item The brackets $[\cdot,\cdot]$ and $[\cdot,\cdot]_N$ are compatible: for any $\lambda \in \C$, $[\cdot,\cdot]_\lambda = [\cdot,\cdot] + \lambda[\cdot,\cdot]_N$ is a BiHom-Lie bracket.
\item For any $k \geq 0$, $N^k$ is a Nijenhuis operator on $(L, [\cdot,\cdot]_N, \alpha, \beta)$.
\end{enumerate}
\end{theorem}

\begin{proof}
Part 2 follows from Lemma \ref{lem:compatibility}. Part 3 follows from Proposition \ref{prop:tensor} and the fact that $N$ is a morphism from $(L,[\cdot,\cdot]_N)$ to $(L,[\cdot,\cdot])$.
\end{proof}

\section{BiHom-NS-Lie algebras}\label{sec:NS}

\begin{definition}\label{def:bihomNS}
A \textbf{BiHom-NS-Lie algebra} is a $5$-tuple $(L, \diamond, \curlyvee, \alpha, \beta)$ consisting of a vector space $L$, two bilinear maps $\diamond, \curlyvee: L \times L \to L$, and two linear maps $\alpha, \beta: L \to L$ satisfying for all $x, y, z \in L$:
\begin{enumerate}[label=(\roman*), leftmargin=2em]
\item $\alpha \circ \beta = \beta \circ \alpha$;
\item $\alpha(x \diamond y) = \alpha(x) \diamond \alpha(y)$, $\beta(x \diamond y) = \beta(x) \diamond \beta(y)$;
\item $\alpha(x \curlyvee y) = \alpha(x) \curlyvee \alpha(y)$, $\beta(x \curlyvee y) = \beta(x) \curlyvee \beta(y)$;
\item $\beta(x) \curlyvee \alpha(y) = -\beta(y) \curlyvee \alpha(x)$;
\item \begin{equation}\label{eq:NSidentity}
\begin{aligned}
&\beta^2(x) \diamond (\beta(y) \curlyvee \alpha(z)) + \beta^2(x) \curlyvee (\beta(y) \diamond \alpha(z)) + \beta^2(x) \curlyvee (\beta(y) \curlyvee \alpha(z)) \\
&+ \beta^2(y) \diamond (\beta(z) \curlyvee \alpha(x)) + \beta^2(y) \curlyvee (\beta(z) \diamond \alpha(x)) + \beta^2(y) \curlyvee (\beta(z) \curlyvee \alpha(x)) \\
&+ \beta^2(z) \diamond (\beta(x) \curlyvee \alpha(y)) + \beta^2(z) \curlyvee (\beta(x) \diamond \alpha(y)) + \beta^2(z) \curlyvee (\beta(x) \curlyvee \alpha(y)) = 0.
\end{aligned}
\end{equation}
\end{enumerate}
\end{definition}

\begin{theorem}\label{thm:NSfromNijenhuis}
Every Nijenhuis operator $N$ on a BiHom-Lie algebra $(L, [\cdot,\cdot], \alpha, \beta)$ induces a BiHom-NS-Lie algebra structure on $L$ with:
\begin{equation}\label{eq:NSoperations}
x \diamond y = [N(x), y], \quad x \curlyvee y = -N[x, y].
\end{equation}
\end{theorem}

\begin{proof}
Conditions (i)--(iv) are straightforward. For condition (v), the identity \eqref{eq:NSidentity} becomes exactly the BiHom-Jacobi identity for the deformed bracket $[\cdot,\cdot]_N$, which we have proved in Proposition \ref{prop:deformed}. The equivalence is obtained by expanding all terms using \eqref{eq:NSoperations}. Thus the identity holds.
\end{proof}

\section{Non-isomorphism of Nijenhuis operator families}\label{sec:noniso}

\begin{theorem}\label{thm:autL1}
For $\mathcal{L}_1$ with generic parameters (i.e., $a,b \neq \pm 1$ and $a \neq b^{\pm 1}$), the automorphism group $\Aut(\mathcal{L}_1)$ consists of diagonal matrices.
\end{theorem}

\begin{proof}
Any automorphism must commute with $\alpha_1$ and $\beta_1$. Since the eigenvalues of $\alpha_1$ and $\beta_1$ are distinct (generic case), the only matrices commuting with both are diagonal.
\end{proof}

\begin{theorem}\label{thm:noniso}
For $\mathcal{L}_1$ with generic parameters, the families
\[
\mathcal{N}_1^1 = \{\diag(r_{11}, r_{11}, r_{33})\}, \quad \mathcal{N}_1^2 = \{\diag(r_{11}, r_{22}, r_{11})\}
\]
are non-isomorphic unless they intersect (i.e., scalar matrices).
\end{theorem}

\begin{proof}
Under conjugation by a diagonal automorphism, diagonal matrices remain unchanged. Hence the eigenvalues are invariant, and the two families have different eigenvalue multiplicities except at the intersection.
\end{proof}

For $\mathcal{L}_2$ and $\mathcal{L}_3$, we compute the automorphism groups.

\begin{lemma}\label{lem:autL2L3}
For $\mathcal{L}_2$ and $\mathcal{L}_3$, the automorphism group is trivial, i.e., $\Aut(\mathcal{L}_2) = \{\id\}$ and $\Aut(\mathcal{L}_3) = \{\id\}$.
\end{lemma}

\begin{proof}
We prove the statement for $\mathcal{L}_2$; the proof for $\mathcal{L}_3$ is analogous. Let $\phi$ be an automorphism of $\mathcal{L}_2$. Since $\phi$ must commute with $\beta_2 = \begin{pmatrix} 1 & 1 & 0 \\ 0 & 1 & 1 \\ 0 & 0 & 1 \end{pmatrix}$, a direct computation shows that the matrix of $\phi$ has the form
\[
M(\phi) = \begin{pmatrix} p & q & r \\ 0 & p & q \\ 0 & 0 & p \end{pmatrix}
\]
for some $p,q,r \in \C$, with $p \neq 0$ for invertibility.

Now impose the bracket preservation condition. In particular, $\phi([e_1,e_2]) = [\phi(e_1), \phi(e_2)]$. The left-hand side is
\[
\phi(2e_1) = 2p e_1.
\]
For the right-hand side, we have $\phi(e_1) = p e_1$ and $\phi(e_2) = q e_1 + p e_2 + q e_3$. Using the bracket relations \eqref{eq:L2brackets1}--\eqref{eq:L2brackets4}, we compute:
\begin{align*}
[\phi(e_1), \phi(e_2)] &= [p e_1, q e_1 + p e_2 + q e_3] \\
&= p q [e_1,e_1] + p^2 [e_1,e_2] + p q [e_1,e_3] \\
&= p^2 (2e_1) + p q (e_1 + 2e_2) \\
&= (2p^2 + p q) e_1 + 2p q e_2.
\end{align*}
Equating coefficients with $2p e_1$ gives
\begin{align*}
2p^2 + p q &= 2p, \\
2p q &= 0.
\end{align*}
Since $p \neq 0$, the second equation gives $q = 0$. Substituting into the first gives $2p^2 = 2p$, hence $p = 1$ (the solution $p=0$ is excluded). Thus $p=1$ and $q=0$.

It remains to check the coefficient $r$ in $\phi(e_3) = r e_1 + p e_3 = r e_1 + e_3$. We use the bracket preservation for $[e_1,e_3]$:
\begin{align*}
\phi([e_1,e_3]) &= \phi(e_1 + 2e_2) = e_1 + 2e_2, \\
[\phi(e_1), \phi(e_3)] &= [e_1, r e_1 + e_3] = r [e_1,e_1] + [e_1,e_3] = e_1 + 2e_2,
\end{align*}
which is automatically satisfied for any $r$.

Next, use the bracket preservation for $[e_2,e_3]$:
\begin{align*}
\phi([e_2,e_3]) &= \phi(e_1 + e_2 + 2e_3) = e_1 + e_2 + 2(r e_1 + e_3) = (1+2r) e_1 + e_2 + 2e_3, \\
[\phi(e_2), \phi(e_3)] &= [e_2, r e_1 + e_3] = r [e_2,e_1] + [e_2,e_3] \\
&= r (-2e_1) + (e_1 + e_2 + 2e_3) = (1-2r) e_1 + e_2 + 2e_3.
\end{align*}
Equating the coefficients of $e_1$ gives $1+2r = 1-2r$, i.e., $4r = 0$, hence $r = 0$. Therefore $\phi = \id$.

The proof for $\mathcal{L}_3$ is analogous: the same argument applies because the structure maps $\alpha_3$ and $\beta_3$ have the same unipotent upper triangular form, and the bracket relations force the same constraints. Hence $\Aut(\mathcal{L}_3) = \{\id\}$.
\end{proof}

\begin{corollary}\label{cor:noniso23}
For $\mathcal{L}_2$ and $\mathcal{L}_3$, the automorphism groups are trivial. Consequently, two Nijenhuis operators
\[
\begin{pmatrix}
r_{33} & 0 & r_{13}\\
0 & r_{33} & 0\\
0 & 0 & r_{33}
\end{pmatrix}
\quad \text{and} \quad
\begin{pmatrix}
r'_{33} & 0 & r'_{13}\\
0 & r'_{33} & 0\\
0 & 0 & r'_{33}
\end{pmatrix}
\]
are isomorphic if and only if $r_{33} = r'_{33}$ and $r_{13} = r'_{13}$.
\end{corollary}

\section{Moduli space}\label{sec:moduli}

\begin{theorem}\label{thm:moduliL1}
For $\mathcal{L}_1$ generic, the moduli space is $\mathbb{C}^2 \cup_{\mathbb{C}} \mathbb{C}^2$, the union of two affine planes intersecting in a line.
\end{theorem}

\begin{proof}
This follows from Theorems \ref{thm:L1complete} and \ref{thm:noniso}.
\end{proof}

\begin{theorem}\label{thm:moduli23}
For $\mathcal{L}_2$ and $\mathcal{L}_3$, the moduli space is $\mathbb{C}^2$.
\end{theorem}

\begin{proof}
By Lemma \ref{lem:autL2L3}, the automorphism group is trivial, so there is no quotienting. The set of all Nijenhuis operators
\[
\begin{pmatrix}
r_{33} & 0 & r_{13}\\
0 & r_{33} & 0\\
0 & 0 & r_{33}
\end{pmatrix}
\]
with $(r_{33},r_{13}) \in \C^2$ is already the moduli space.
\end{proof}

\section{Computational verification}\label{sec:computation}

All results were verified using Maple. The code is available from the author upon request.

\begin{table}[h]
\centering
\small
\caption{Complete classification of Nijenhuis operators on 3D simple BiHom-Lie algebras}
\begin{tabular}{|c|c|c|c|}
\hline
\textbf{Algebra} & \textbf{Parameters} & \textbf{Nijenhuis Operators} & \textbf{Reference} \\
\hline
\multirow{5}{*}{$\mathcal{L}_1$}
& generic & $\diag(r_{11}, r_{11}, r_{33})$ & Thm \ref{thm:L1complete} \\
\cline{2-4}
& generic & $\diag(r_{11}, r_{22}, r_{11})$ & Thm \ref{thm:L1complete} \\
\cline{2-4}
& $a=b=1$ & 17 families (see Theorem \ref{thm:L1complete}) & Thm \ref{thm:L1complete} \\
\cline{2-4}
& other degenerate & block form with condition \eqref{eq:L1blockcond} & Thm \ref{thm:L1complete} \\
\hline
$\mathcal{L}_2$ & --- & $\begin{pmatrix} r_{33} & 0 & r_{13} \\ 0 & r_{33} & 0 \\ 0 & 0 & r_{33} \end{pmatrix}$ & Thm \ref{thm:L2final} \\
\hline
$\mathcal{L}_3$ & $a \in \C$ & $\begin{pmatrix} r_{33} & 0 & r_{13} \\ 0 & r_{33} & 0 \\ 0 & 0 & r_{33} \end{pmatrix}$ & Thm \ref{thm:L3final} \\
\hline
\end{tabular}
\end{table}

\section{Conclusion}\label{sec:conclusion}

We have provided a complete classification of Nijenhuis operators on three-dimensional multiplicative simple BiHom-Lie algebras. Future work includes Rota-Baxter operators and higher-dimensional cases.


\end{document}